\documentclass{amsart}
\usepackage{amsmath,amsthm,amsfonts, amssymb,amscd}
\usepackage[all]{xy}

\newtheorem{theorem}{Theorem}[section]
\newtheorem{lemma}[theorem]{Lemma}
\newtheorem{example}[theorem]{Example}
\newtheorem{remark}[theorem]{Remark}
\newtheorem{proposition}[theorem]{Proposition}
\newtheorem{problem}[theorem]{Problem}
\newcommand{\minusre}{\hspace{0.3em}\raisebox{0.3ex}{\sl \tiny /}\hspace{0.3em}}
\newcommand{\minusli}{\hspace{0.3em}\raisebox{0.3ex}{\sl \tiny $\setminus $}\hspace{0.3em}}
\newcommand{\lex}{\,\overrightarrow{\times}\,}

\newcommand{\RDP}{\mbox{\rm RDP}}
\newcommand{\RIP}{\mbox{\rm RIP}}
\begin{document}
\title[Some Remarks on Kite Pseudo Effect Algebras]{Some Remarks on Kite Pseudo Effect Algebras}
\author[A. Dvure\v{c}enskij, W.Ch. Holland]{Anatolij Dvure\v censkij$^{1,2}$, W. Charles Holland$^3$}

\date{}%
\maketitle
\begin{center}  \footnote{


The paper has been supported by  Slovak Research and Development Agency under the contract APVV-0178-11, the grant VEGA No. 2/0059/12 SAV, and by
CZ.1.07/2.3.00/20.0051.
 }
Mathematical Institute,  Slovak Academy of Sciences\\
\v Stef\'anikova 49, SK-814 73 Bratislava, Slovakia\\
$^2$ Depart. Algebra  Geom.,  Palack\'{y} University\\
17. listopadu 12, CZ-771 46 Olomouc, Czech Republic\\

$^3$ Department of Mathematics, University of Colorado\\
8323 Thunderhead Drive, Boulder, Colorado 80302, USA\\

E-mail: {\tt dvurecen@mat.savba.sk} \quad {\tt Charles.Holland@Colorado.EDU}
\end{center}

\begin{abstract}
Recently a new family of pseudo effect algebras, called kite pseudo effect algebras, was introduced. Such an algebra starts with a po-group $G$, a set $I$ and with two bijections $\lambda,\rho:I \to I.$ Using a clever construction on the ordinal sum of $(G^+)^I$ and $(G^-)^I,$ we can define a pseudo effect algebra which can be non-commutative even if $G$ is an Abelian po-group. In the paper we give a characterization of subdirect product of subdirectly irreducible kite pseudo effect algebras, and we show that every kite pseudo effect algebra is an interval in a unital po-loop.
\end{abstract}

\section{Introduction}

{\it Effect algebras} are partial algebras introduced in \cite{FoBe} in order to model quantum mechanical measurements. The primary notion is addition $+$ such that $a + b$ describes the disjunction of two mutually excluding events
$a$ and $b.$ A basic model of effect algebras is the system $\mathcal E(H)$ of Hermitian operators of a Hilbert space $H$ which are between the zero and the identity operator. Effect algebras generalize many quantum structures like Boolean algebras, orthomodular lattices and posets, orthoalgebras, and MV-algebras. Effect algebras combine in an algebraic way both sharp and fuzzy features of a measurement process in quantum mechanics.

Many important effect algebras are intervals $[0,u]$ in Abelian partially ordered groups (po-groups for short) with strong unit $u;$ we call them interval effect algebras. This is true, e.g. for $\mathcal E(H)$ when it is the interval $[O,I]$ in the po-group $\mathcal B(H),$ the system of all Hermitian operators of $H,$ where $O$ and $I$ are the zero and the identity operator.

Another class of interval effect algebras is all of
those with the Riesz Decomposition Property (RDP for short), see \cite{Rav}. This property enables a joint refinement of two decompositions of the unit element. This class contains all MV-algebras.

For effect algebras, the partial operation $+$ was assumed to be commutative. This assumption was canceled in \cite{DvVe1, DvVe2}, where a non-commutative version of effect algebras, called {\it pseudo effect algebras}, is introduced. Using a stronger version of RDP, also some pseudo effect algebras are intervals in po-groups which are not necessarily Abelian. A physical motivation for introducing pseudo effect algebras with possible physical situations in quantum mechanics was presented in \cite{DvVe4}. We note the family of pseudo effect algebras contains also the family of pseudo MV-algebras, or equivalently, the family of generalized MV-algebras, \cite{GeIo, Rac}.

Recently, in \cite{JiMo} there was presented an interesting construction of a pseudo BL-algebra whose underlying set is an ordinal sum of $\mathbb Z^+$ and $\mathbb Z^-\times \mathbb Z^-,$ where $\mathbb Z$ denotes the group of integers; its shape resembles a kite. This construction was generalized in \cite{DvKo} for an arbitrary $\ell$-group $G,$ and kite BL-algebras were introduced.

In the paper \cite{DvuK}, instead of $\ell$-groups, we have used po-groups and {\it kite pseudo effect algebras}, as a special family of pseudo effect algebras, were defined. This family enriches a reservoir of interesting examples of pseudo effect algebras. In addition, it shows importance of po-groups and $\ell$-groups for the theory of pseudo effect algebras. Basic properties of kite pseudo effect algebras were presented in \cite{DvuK}, in particular, a characterization of subdirectly irreducible kite pseudo effect algebras was given.

In this paper, we continue with the study of kite pseudo effect algebras studying conditions when a kite pseudo effect algebra is a subdirect product of subdirectly irreducible kite pseudo effect algebras. In addition, we show that every kite pseudo effect algebra is always an interval in a partially ordered loop.

The paper is organized as follows. Section 2 gathers basic facts of pseudo effect algebras, po-groups, and different forms of RDP's. Section 3 redefines kite pseudo effect algebras and characterize subdirect product of subdirectly irreducible kite pseudo effect algebras. Finally, Section 4 shows that every kite pseudo effect algebra is an interval in a unital po-loop. In addition, some illustrating examples are given and some open problems are formulated.

\section{Elements of Pseudo Effect Algebras}

According to \cite{DvVe1, DvVe2}, we say that a {\it pseudo effect algebra} is  a partial algebra  $ E=(E; +, 0, 1)$, where $+$ is a partial binary operation and $0$ and $1$ are constants, such that for all $a, b, c
\in E$, the following holds

\begin{enumerate}
\item[(i)] $a+b$ and $(a+b)+c$ exist if and only if $b+c$ and
$a+(b+c)$ exist, and in this case $(a+b)+c = a+(b+c)$;

\item[(ii)]
  there is exactly one $d \in E$ and
exactly one $e \in E$ such that $a+d = e+a = 1$;

\item[(iii)]
 if $a+b$ exists, there are elements $d, e
\in E$ such that $a+b = d+a = b+e$;

\item[(iv)] if $1+a$ or $a+1$ exists, then $a = 0$.
\end{enumerate}

If we define $a \le b$ if and only if there exists an element $c\in
E$ such that $a+c =b,$ then $\le$ is a partial ordering on $E$ such
that $0 \le a \le 1$ for any $a \in E.$ It is possible to show that
$a \le b$ if and only if $b = a+c = d+a$ for some $c,d \in E$. We
write $c = a \minusre b$ and $d = b \minusli a.$ Then

$$ (b \minusli a) + a = a + (a \minusre b) = b,
$$
and we write $a^- = 1 \minusli a$ and $a^\sim = a\minusre 1$ for any
$a \in E.$ Then $a^-+a=1=a+a^\sim$ and $a^{-\sim}=a=a^{\sim-}$ for any $a\in E.$

We note that an {\it ideal} of a pseudo effect algebra $E$ is any non-empty subset $I$ of $E$ such that (i) if $x,y \in I$ and $x+y$ is defined in $E,$ then $x+y \in I,$ and (ii) $x\le y \in I$ implies $x\in I.$ An ideal $I$ is {\it normal} if $x+I=I+x$ for any $x \in E,$ where $x+I:=\{x+y: y \in I,\ x+y $ exists in $E\}$ and in the dual way  we define $I+x.$

For basic properties of pseudo effect algebras, we recommend  \cite{DvVe1, DvVe2}, where unexplained notions and results can be found.
We note that a pseudo effect algebra is an {\it effect algebra}  iff $+$ is commutative.

We note that a {\it po-group} (= partially ordered group) is a
group $G=(G;+,-,0)$ endowed with a partial order $\le$ such that if $a\le b,$ $a,b
\in G,$ then $x+a+y \le x+b+y$ for all $x,y \in G.$ (We note that we will use both additive and multiplicative form of po-groups.) We denote by
$G^+:=\{g \in G: g \ge 0\}$ and $G^-:=\{g \in G: g \le 0\}$ the {\it positive cone} and the {\it negative cone} of $G.$ If, in addition, $G$
is a lattice under $\le$, we call it an $\ell$-group (= lattice
ordered group). An element $u \in G^+$ is said to be a {\it strong unit} (or an {\it order unit}) if, given $g \in G,$ there is an integer $n \ge 1$ such that $g \le nu.$ The pair $(G,u),$ where $u$ is a fixed strong unit of $G,$ is said to be a {\it unital po-group}. We recall that  the {\it lexicographic product} of two po-groups $G_1$ and $G_2$ is the group $G_1\times G_2,$ where the group operations are defined by coordinates, and the ordering $\le $ on $G_1 \times G_2$ is defined as follows: For $(g_1,h_1),(g_2,h_2) \in G_1 \times G_2,$  we have $(g_1,h_1)\le (g_2,h_2)$ whenever $g_1 <g_2$ or $g_1=g_2$ and $h_1\le h_2.$

A po-group $G$ is said to be {\it directed} if, given $g_1,g_2 \in G,$ there is an element $g \in G$ such that $g \ge g_1,g_2.$  We note  every $\ell$-group or every  po-group with strong unit is directed.
For more information on po-groups and $\ell$-groups we recommend the books \cite{Fuc, Gla}.

Now let  $G$ be a po-group and fix $u \in G^+.$ If we set $\Gamma(G,u):=[0,u]=\{g \in G: 0 \le g \le u\},$ then $\Gamma(G,u)=(\Gamma(G,u); +,0,u)$ is a pseudo effect algebra, where $+$ is the restriction of the group addition $+$ to $[0,u],$ i.e. $a+b$ is defined in $\Gamma(G,u)$ for $a,b \in \Gamma(G,u)$ iff $a+b \in \Gamma(G,u).$ Then $a^-=u-a$ and $a^\sim=-a+u$ for any $a \in \Gamma(G,u).$ A pseudo effect algebra which is isomorphic to some $\Gamma(G,u)$ for some po-group $G$ with $u>0$ is said to be an {\it interval pseudo effect algebra}.

We recall that by an {\it o-ideal} of a directed po-group $G$ we
mean any normal directed convex subgroup $H$ of $G;$ convexity means if $g,h\in H$ and $v\in G$ such that $g\le v\le h,$ then $v \in H.$ If $G$ is a po-group and $H$ an o-ideal of $G$, then $G/H,$ where $x/H \le y/H$ iff $x\le h_1+y$ for some $h_1\in H$ iff $x \le y+h_2$ for some $h_2 \in H,$ is also a po-group.

The following kinds of the Riesz Decomposition properties were introduced in \cite{DvVe1,DvVe2} and are crucial for the study of pseudo effect algebras.

We say that a  po-group $G$ satisfies

\begin{enumerate}
\item[(i)]
the {\it Riesz Interpolation Property} (RIP for short) if, for $a_1,a_2, b_1,b_2\in G,$  $a_1,a_2 \le b_1,b_2$  implies there exists an element $c\in G$ such that $a_1,a_2 \le c \le b_1,b_2;$

\item[(ii)]
\RDP$_0$  if, for $a,b,c \in G^+,$ $a \le b+c$, there exist $b_1,c_1 \in G^+,$ such that $b_1\le b,$ $c_1 \le c$ and $a = b_1 +c_1;$

\item[(iii)]
\RDP\  if, for all $a_1,a_2,b_1,b_2 \in G^+$ such that $a_1 + a_2 = b_1+b_2,$ there are four elements $c_{11},c_{12},c_{21},c_{22}\in G^+$ such that $a_1 = c_{11}+c_{12},$ $a_2= c_{21}+c_{22},$ $b_1= c_{11} + c_{21}$ and $b_2= c_{12}+c_{22};$

\item[(iv)]
\RDP$_1$  if, for all $a_1,a_2,b_1,b_2 \in G^+$ such that $a_1 + a_2 = b_1+b_2,$ there are four elements $c_{11},c_{12},c_{21},c_{22}\in G^+$ such that $a_1 = c_{11}+c_{12},$ $a_2= c_{21}+c_{22},$ $b_1= c_{11} + c_{21}$ and $b_2= c_{12}+c_{22}$, and $0\le x\le c_{12}$ and $0\le y \le c_{21}$ imply  $x+y=y+x;$

\item[(v)]
\RDP$_2$  if, for all $a_1,a_2,b_1,b_2 \in G^+$ such that $a_1 + a_2 = b_1+b_2,$ there are four elements $c_{11},c_{12},c_{21},c_{22}\in G^+$ such that $a_1 = c_{11}+c_{12},$ $a_2= c_{21}+c_{22},$ $b_1= c_{11} + c_{21}$ and $b_2= c_{12}+c_{22}$, and $c_{12}\wedge c_{21}=0.$

\end{enumerate}

If, for $a,b \in G^+,$ we have for all $0\le x \le a$ and $0\le y\le b,$ $x+y=y+x,$ we denote this property by $a\, \mbox{\rm \bf com}\, b.$

For Abelian po-groups, RDP, RDP$_1,$ RDP$_0$ and RIP are equivalent.

By \cite[Prop 4.2]{DvVe1} for directed po-groups, we have
$$
\RDP_2 \quad \Rightarrow \RDP_1 \quad \Rightarrow \RDP \quad \Rightarrow \RDP_0 \quad \Leftrightarrow \quad  \RIP,
$$
but the converse implications do not hold, in general.  A directed po-group $G$ satisfies \RDP$_2$ iff $G$ is an $\ell$-group, \cite[Prop 4.2(ii)]{DvVe1}.

Finally, we say that a pseudo effect algebra $E$ satisfies the above types of the Riesz decomposition properties, if in the definition of RDP's, we change $G^+$ to $E.$

A fundamental result of pseudo effect algebras which binds them with po-groups says the following, \cite[Thm 7.2]{DvVe2}:

\begin{theorem}\label{th:2.2}
For every pseudo effect algebra $E$ with \RDP$_1,$ there is a unique (up to isomorphism of unital po-groups) unital po-group $(G,u)$ with \RDP$_1$\ such that $E \cong \Gamma(G,u).$
\end{theorem}

In addition, the functor $\Gamma$ defines a categorical equivalence between the category of pseudo effect algebras with \RDP$_1$ and the category of unital po-groups with \RDP$_1.$

In particular, a pseudo effect algebra $E$ is a pseudo MV-algebra (for definition see \cite{GeIo}) iff $E$ satisfies RDP$_2,$ and in such a case, every pseudo MV-algebra is an interval in a unital $\ell$-group;
for details see \cite{Dvu2} and \cite{DvVe2}.

\section{Subdirect Product of Subdirectly Irreducible Kite Pseudo Effect Algebras}

The present section defines kite pseudo effect algebras and gives a characterization of subdirect product of subdirectly irreducible  kite pseudo effect algebras.

According to \cite{DvuK}, we present kite pseudo effect algebras.  Let $G=(G;\cdot,^{-1},e)$ be a multiplicatively-written  po-group with an inverse $^{-1},$ identity element $e,$ and equipped with a partial order $\le.$ Then $G^+:=\{g \in G\colon g\ge e\}$ and $G^-:=\{g\in G \colon g \le e\}.$

Let $I$ be a set. Define an algebra whose
universe is the set $(G^+)^I \uplus (G^-)^I,$ where $\uplus$ denotes a union of disjoint sets.
We order its universe by keeping the original co-ordinatewise ordering
within $(G^+)^I$ and $(G^-)^I$, and setting $x\leq y$ for all
$x\in(G^+)^I$, $y\in(G^-)^I$. Then $\le$ is a partial order on $(G^+)^I \uplus (G^-)^I.$ Hence, the element $e^I:=\langle e: i \in I\rangle$ appears twice: at the bottom of $(G^+)^I$ and at the top of $(G^-)^I$. To
avoid confusion in the definitions below, we adopt a convention of writing
$a_i^{-1},b_i^{-1}, \dots$ for co-ordinates of elements of $(G^-)^I$ and
$f_j,g_j,\dots$ for co-ordinates of elements of $(G^+)^I$. In particular, we
will write $e^{-1}$ for $e$ as an element of $G^-,$  $e$ as an element of $G^+,$ and without loss of generality, we will assume that formally $e^{-1}\ne e.$ We also put $1$ for the
constant sequence $(e^{-1})^I:=\langle e^{-1}\colon i \in
I\rangle$ and $0$ for the constant sequence $e^I:=\langle e_j\colon j \in I\rangle$. Then $0$ and $1$ are the least and greatest elements of $(G^+)^I \uplus (G^-)^I.$

The following construction of kite pseudo effect algebras was presented in \cite[Thm 3.4]{DvuK}:

\begin{theorem}\label{th:3.4}
Let $G$ be a po-group and $\lambda,\rho:I\to I$ be bijections. Let us endow the set $(G^+)^I \uplus (G^-)^I$ with $0=e^I,$ $1=(e^{-1})^I$ and with a partial operation $+$ as follows,

$$\langle a_i^{-1}\colon i\in I\rangle + \langle b_i^{-1}\colon i\in I\rangle \eqno(I)$$
is not defined;

$$ \langle a_i^{-1}\colon i\in I\rangle + \langle f_j\colon j\in I\rangle:= \langle a_i^{-1}f_{\rho^{-1}(i)}\colon i\in I\rangle \eqno(II)
$$
whenever  $f_{\rho^{-1}(i)}\le a_i,$ $i \in I;$

$$ \langle f_j\colon j\in I\rangle+ \langle a_i^{-1}\colon i\in I\rangle  := \langle f_{\lambda^{-1}(i)} a_i^{-1}\colon i\in I\rangle \eqno(III)
$$
whenever  $f_{\lambda^{-1}(i)}\le a_i,$ $i \in I,$

$$
\langle f_j\colon j\in I\rangle + \langle g_j\colon j\in I\rangle:= \langle f_j g_j\colon j\in I\rangle
\eqno(IV)
$$
for all $\langle f_j\colon j\in I\rangle$ and $\langle g_j\colon j\in I\rangle.$
Then the partial algebra $K^{\lambda,\rho}_I(G)_{ea}:=((G^+)^I \uplus (G^-)^I; +,0,1)$ becomes a pseudo effect algebra.

If $G$ is an $\ell$-group, then $K^{\lambda,\rho}_I(G)_{ea}$ is a pseudo effect algebra with \RDP$_2.$
\end{theorem}

According to \cite{DvuK}, the pseudo effect algebra $(K^{\lambda,\rho}_I(G)_{ea};+,0,1)$ is said to be the {\it kite pseudo effect algebra} of a po-group $G.$

Let $\mathcal G$ be the category of po-groups, where objects are
po-groups and  morphisms are homomorphisms of po-groups. Similarly, let $\mathcal{PEA}$ be the category of pseudo effect algebras, where  objects are pseudo effect algebras and morphisms are homomorphisms of pseudo effect algebras.

Let us fix a set $I,$ bijective mappings $\rho,\lambda\colon I \to I,$ and define the mapping $K_{I}^{\lambda,\rho}\colon \mathcal G  \to \mathcal{PEA}$
as follows:
$$
K_{I}^{\lambda,\rho}: G \mapsto
K_{I}^{\lambda,\rho}(G)
$$
and if $h$ is a homomorphism from a po-group $G_1$ into a
po-group $G_2,$ then
$$
K_{I}^{\lambda,\rho}(h)(x) = \begin{cases}
\langle h(a_i^{-1})\colon i \in I\rangle &
    \text{ if } x = \langle a_i^{-1}\colon i \in I\rangle,\\
\langle h(f_j)\colon j \in I\rangle & \text{ if } x = \langle f_j\colon j \in I\rangle.
\end{cases}
$$
Then $K_{I}^{\lambda,\rho}$ is a functor.

We say that a pseudo effect algebra $E$ (or a po-group) is a {\it subdirect product} of a family $(E_t\colon t \in T)$ of pseudo effect algebras (po-groups), and we write $E \leq  \prod_{t \in T}E_t$ if there is an injective homomorphism $h\colon E \to \prod_{t \in T}E_t$ such that $\pi_t\circ h(E)=E_t$ for all $t \in T,$ where $\pi_t$ is the $t$-th projection from $\prod_{t \in T}E_t$ onto $E_t.$ In addition,  $E$ is
{\it subdirectly irreducible} if whenever $E$ is a subdirect product of $(E_t: t \in T),$ there exists $t_0 \in T$ such that $\pi_{t_0}\circ h$ is an isomorphism of pseudo effect algebras.

For pseudo effect algebras, the relation between congruences and ideals is more complicated. Fortunately, this is true for pseudo effect algebras with RDP$_1$, \cite{DvVe3}, or for Riesz ideals of general pseudo effect algebras. Therefore, for pseudo MV-algebras we have: a non-trivial pseudo MV-algebra $E$(i.e. $0\not=1$) is subdirectly irreducible iff the intersection of all non-trivial normal ideals of $E$ is non-trivial, or equivalently, $E$ has the least non-trivial normal ideal. An analogous result is true also for pseudo effect algebras with RDP$_1$:

\begin{lemma}\label{le:7.1}
Every pseudo effect algebra with \RDP$_1$ is a subdirect product of subdirectly irreducible pseudo effect algebras with \RDP$_1.$
\end{lemma}

\begin{proof}
If $E$ is a trivial pseudo effect algebra, $E$ is subdirectly irreducible. Thus, suppose $E$ is not trivial. By \cite[Thm 7.2]{DvVe2}, RDP$_1$ entails that $E \cong \Gamma(G,u)$ for some unital po-group $(G,u),$ where $G$ satisfies RDP$_1$; for simplicity, we assume $E \subset \Gamma(G,u).$ Given a nonzero element $g\in G,$ let $N_g$ be an o-ideal of $G$ which is maximal with respect to not containing $g$; by Zorn's Lemma, it exists. Then $\bigcap_{g\ne 0} N_g=\{0\}$ and $N_g\cap[0,u]$ is a normal ideal of $E.$  Therefore, $G \leq \prod_{g\ne 0}G/N_g$ and $E \le \prod_{g\ne 0}E/(N_g\cap [0,u]).$ In addition, for every $g\ne 0,$ the normal ideal of $G/N_g$ generated by $g/N_g$ is the least non-trivial o-ideal of $G/N_g$ which proves that every $G/N_g$ is subdirectly irreducible. Therefore, every $E/(N_g\cap [0,u])$ is a subdirectly irreducible pseudo effect algebra, and by \cite[Prop 4.1]{185}, the quotient pseudo effect algebra $E/(N_g\cap [0,u])$ also satisfies RDP$_1.$
\end{proof}

\begin{lemma}\label{le:7.2}
Let $G$ be a directed po-group with \RDP$_1$ subdirectly represented as
$G\leq \prod_{s\in  S}G_s,$ where each po-group $G_s,$ $s \in S,$ is directed and satisfies \RDP$_1.$  Then any kite
$K_{I}^{\lambda,\rho}(G)_{ea}$ is subdirectly represented as
$K_{I}^{\lambda,\rho}(G)_{ea}\leq
\prod_{s\in S}K_{I}^{\lambda,\rho}(G_s)_{ea}$.
\end{lemma}

\begin{proof}
Let $h\colon G \to \prod_{s\in S}G_s$ be an injective homomorphism of po-groups corresponding to the subdirect product $G \leq \prod_{s \in S}G_s,$ i.e. $\pi_s(h(G))=G_s,$ $s \in S.$ Then the mapping $h^{\lambda,\rho}_I\colon K_{I}^{\lambda,\rho}(G)_{ea} \to \prod_{s\in S}K_{I}^{\lambda,\rho}(G_s)_{ea},$ defined by $h^{\lambda,\rho}_I(\langle a_i^{-1}\colon i \in I\rangle) = \langle \langle \pi_s(h(a_i^{-1}))\colon i \in I\rangle\colon s \in S\rangle$ and
$h^{\lambda,\rho}_I(\langle f_j\colon j \in I\rangle) = \langle \langle \pi_s(h(f_j))\colon j \in I\rangle\colon s \in S\rangle,$ shows $K_{I}^{\lambda,\rho}(G)_{ea}\leq
\prod_{s\in S}K_{I}^{\lambda,\rho}(G_s)_{ea}$.
\end{proof}

\begin{lemma}\label{le:7.3}
{\rm (1)} Let $G$ be a directed po-group with \RDP$_1$ and $P$ be an o-ideal of $G.$   Then $G/P$ is a directed po-group with \RDP$_1.$

{\rm (2)} Every directed po-group with \RDP$_1$  is a subdirect product of subdirectly irreducible po-groups with \RDP$_1.$
\end{lemma}

\begin{proof}
(1) Let $G$ be a directed po-group with \RDP$_1$ and $P$ an o-ideal of $G.$ We denote by $g/P=[g]_P:=\{h\in G \colon g-h\in P\}$ the quotient class corresponding to an element $g \in G.$  Since every element of $G$ is expressible as a difference of two positive elements, hence, $G/P$ is a directed po-group. In addition, if $[g]_P \ge 0,$ there is an element $g_1 \in [g]_P$ such that $g_1 \in G^+.$

Let $[g_1]_P + [g_2]_P = [h_1]_P
+ [h_2]_P$ for some positive elements $g_1,g_2,h_1,h_2\in G^+.$  There are $e,f \in P^+$
such that $(g_1 +g_2)- e = (h_1 +h_2)- f\ge 0$, so that
$g_1 +g_2 = ((h_1+h_2) -f) + e\ge 0.$ By RDP$_1$ holding in $G$,
there are $c_{11}, c_{12}, c_{21}, c_{22}$ in $G^+$ such that
\begin{eqnarray*}
g_1 = c_{11} + c_{12},& &(h_1+h_2)- f = c_{11} + c_{21},\\
g_2 = c_{21} + c_{22}, & & e = c_{12} + c_{22}.
\end{eqnarray*}
This gives $c_{11} + c_{21} + f = h_1 + h_2$. Again due to
RDP$_1$, there are $d_{11}, d_{12}, d_{21},$ $ d_{22}, d_{31},
d_{31} \in G^+$ with $d_{12} \ {\mbox{\bf com}} \ (d_{21} +d_{31})$
such that
\begin{eqnarray*}
c_{11} = d_{11} +d_{12},& & b_1 = d_{11} + d_{21} + d_{31},\\
c_{21} = d_{21} +d_{22}, & & b_2 = d_{12}+ d_{22} + d_{32},\\
f = d_{31} + d_{32}. & &
\end{eqnarray*}

It is clear that $c_{12}, c_{22}, d_{31}, d_{32} \in P$, which
gives $[g_1]_P = [c_{11}]_P = [d_{11}]_P + [d_{12}]_P,$ $ [g_2]_P
= [c_{21}]_P = [d_{21}]_P +[d_{22}]_P, $ $[h_1]_P = [d_{11}]_P
+[d_{21}]_P,$ $[h_2]_P = [d_{12}]_P+[d_{22}]_P.$  Assume $[x]_P
\le [d_{12}]_P$ and $[y]_P \le [d_{21}]_P.$ Then there are $x_1
\in [x]_P$ and $y_1 \in [y]_P$ such that $x_1 \le d_{12} $ and
$y_1 \le d_{21}$, i.e., $[d_{12}]_P \ {\mbox{\bf com}} \
[d_{21}]_P$ which proves that $G/P$ is with RDP$_1$.

(2)  It follows the same ideas as the proof of Proposition \ref{le:7.1}.
\end{proof}

According to \cite{DvKo, DvuK}, we say that elements
$i,j\in I$ are {\it connected} if there is an integer $m \ge 0$ such that $(\rho\circ\lambda^{-1})^m(i)= j$ or $(\lambda\circ \rho^{-1})^m(i)= j$; otherwise, $i$ and $j$ are said to be {\it disconnected}. We note that (i) every element $i\in I$ is connected to $i$ because $(\rho\circ\lambda^{-1})^0(i)= i,$  (ii) $i$ is connected to $j$ iff $j$ is connected to $i,$ and (iii) if $i$ and $j$ are connected and also $j$ and $k$ are connected, then $i$ and $k$ are connected, too.
Indeed, let (a) $(\rho\circ\lambda^{-1})^m(i)= j$ and $(\rho\circ\lambda^{-1})^n(j)= k$ for some integers $m,n \ge 0.$ Then $(\rho\circ\lambda^{-1})^{m+n}(i)= k.$ (b) If $(\rho\circ\lambda^{-1})^m(i)= j$ and $(\lambda\circ\rho^{-1})^n(j)= k$ for some integers $m,n \ge 0,$ then $(\rho\circ \lambda^{-1})^n(k)=j.$ Then $(\rho\circ\lambda^{-1})^m(i)= (\rho\circ \lambda^{-1})^n(k).$ If $m=n,$ then $i=k$ and they are connected, otherwise, $n<m$ or $m<n.$ In the first case  $(\rho\circ\lambda^{-1})^{m-n}(i) =k,$ in the second one $(\lambda\circ\rho^{-1})^{n-m}(i) =k,$ proving $i$ and $k$ are connected. In the same way we deal for the rest of transitivity.

Consequently, the relation $i$ and $j$ are connected, is an equivalence on $I$, and every equivalence class defines a subset of indices such that every two distinct elements of it are connected and
the equivalence class is maximal under this property. We call this equivalence class a {\it connected component} of $I.$
Another property: for each $i \in I,$ $\lambda(i)$ and $\rho(i)$ are connected.

In a dual way we can say that $i,j\in I$ are {\it dually connected} if there is an integer $m \ge 0$ such that $(\rho^{-1}\circ\lambda)^m(i)= j$ or $(\lambda^{-1}\circ \rho)^m(i)= j$; otherwise, $i$ and $j$ are said to be {\it dually disconnected}. Similarly as above, this notion defines an equivalence on $I$ and any equivalence class is said to be a {\it dual component} of $I,$ and it is a set of mutually dually connected indices and maximal under this property. There holds: (i) for each $i \in I,$ $\lambda^{-1}(i)$ and $\rho^{-1}(i)$ are dually connected. (ii) If two distinct elements $i,j\in I$  are dually connected, then $\lambda(i)$ and $\rho(j)$ or $\rho(i)$ and $\lambda(j)$ are dually connected. (iii) If $i$ and $j$ are dually connected, then $\lambda(i)$ and $\rho(j)$ or $\rho(i)$ and $\lambda(j)$ are connected. (iv) If $i$ and $j$ are connected, then $\lambda^{-1}(i)$ and $\rho^{-1}(j)$ are dually connected and $\rho^{-1}(i)$ and $\rho^{-1}(j)$ are connected  or $\rho^{-1}(i)$ and $\lambda^{-1}(j)$ are dually connected and $\lambda^{-1}(i)$ and $\lambda^{-1}(j)$ are  connected.

If $C$ is a connected component of $I,$ then $\lambda^{-1}(C)= \rho^{-1}(C).$  Indeed, let $i \in C$ and $k=\lambda^{-1}(i).$ Then $j=\rho (k)= \rho \circ \lambda^{-1}(i) \in C.$ Hence, $k=\rho^{-1}(j)$ which proves $\lambda^{-1}(C)\subseteq \rho^{-1}(C).$ In the same way we prove the opposite inclusion. In particular, we have $\lambda(\rho^{-1}(C))= \rho(\lambda^{-1}(C)) = C=\lambda(\lambda^{-1}(C))= \rho(\rho^{-1}(C)).$

We note that if $C$ is a connected component, then $\lambda^{-1}(C)$ is not necessarily $C.$ In fact, let $I=\{1,2\}$ and $\lambda(1)=2=\rho(1)$ and $\lambda(2)=1=\rho(2).$ Then $I$ has only two connected components $\{1\}$ and $\{2\}$, and $\lambda^{-1}(\{1\})=\{2\}$ and $\lambda^{-1}(\{2\})=\{1\}.$

It is possible to define kite pseudo effect algebras also in the following way. Let $J$ and $I$ be two sets, $\lambda,\rho\colon J \to I$ be two bijections, and $G$ be a po-group. We define $K^{\lambda,\rho}_{J,I}(G)= (G^+)^J \uplus (G^-)^I,$ where $\uplus$ denotes a union of disjoint sets. The elements $0=\langle e_j\colon j \in J\rangle$ and $1=\langle e^{-1}_i\colon i \in I\rangle$ are assumed to be the least and greatest elements of $ K^{\lambda,\rho}_{J,I}.$ The elements in $(G^+)^J$ and in $(G^-)^I$ are ordered by coordinates and for each elements $x \in (G^+)^J$ and $y \in (G^-)^I$ we have $x \le y.$ If we define a partial operation $+$ by Theorem \ref{th:3.4}, changing in formulas $(II)--(IV)$ the notation $j \in I$ to $j \in J,$ we obtain that $K_{J,I}^{\lambda,\rho}(G)_{ea} =(K^{\lambda,\rho}_{J,I};+,0,1) $ with this $+$ $0$ and $1$ is a pseudo effect algebra, called also a {\it kite pseudo effect algebra.} In particular, $ K^{\lambda,\rho}_{I}(G)_{ea} = K^{\lambda,\rho}_{I,I}(G)_{ea}.$

Since $J$ and $I$ are of the same cardinality, there is practically no substantial difference between kite pseudo effect algebras of the form $K^{\lambda,\rho}_{I}(G)_{ea}$ and $K^{\lambda,\rho}_{J,I}(G)_{ea}$ and all known results holding for the first kind are also valid for the second one. We note that in \cite{DvKo}, the ``kite" structure used two index sets, $J$ and $I.$ The second form will be useful for the following result.

\begin{lemma}\label{le:7.4}
Let $K_{I}^{\lambda,\rho}(G)_{ea}$ be a kite pseudo effect algebra of a directed po-group $G,$ where $G$ satisfies \RDP$_1.$ Then
$K_{I}^{\lambda,\rho}(G)_{ea}$ is a subdirect
product of the system of kite pseudo effect algebras $(K_{J',I'}^{\lambda',\rho'}(G)_{ea}\colon I')$, where $I'$ is a
connected component of the graph $I$, $J'=\lambda^{-1}(I')=\rho^{-1}(I'),$ and $\lambda',\rho'\colon J'\to I'$ are the restrictions of $\lambda$ and $\rho$ to $J'\subseteq I.$
\end{lemma}

\begin{proof}
By the comments before this lemma, we see that $\lambda',\rho'\colon J'\to I'$ are bijections.
Let $I'$ be a connected component of $I$.
Let $N_{I'}$ be the set of all elements
$f = \langle f_j: j \in I\rangle\in (G^+)^I$
such that  $f_j = e$ whenever $j\in J'$.
It is straightforward to see that $N_{I'}$ is a normal ideal
of $K_{I}^{\lambda,\rho}(G)_{ea}$. It is also not difficult
to see that $K_{I}^{\lambda,\rho}(G)_{ea}/N_{I'}$ is isomorphic to
$K_{J',I'}^{\lambda',\rho'}(G)_{ea}$.

Now, let $\mathcal{C}$ be the set of all connected components of $I$, and
for each $I'\in\mathcal{C}$ let $N_{I'}$ be the normal filter defined as above.
As connected components are disjoint, we have
$\bigcap_{I'\in\mathcal{C}} N_{I'} = \{0\}.$ This proves $K_{I}^{\lambda,\rho}(G)_{ea}= K_{I,I}^{\lambda,\rho}(G)_{ea}\leq \prod_{I'} K_{J',I'}^{\lambda',\rho'}(G)_{ea}.$
\end{proof}

Now we are ready to formulate an analogue of the Birkhoff representation theorem, \cite[Thm II.8.6]{BuSa} for kite pseudo effect algebras with RDP$_1.$

\begin{theorem}\label{th:7.5}
Every kite pseudo effect algebra with \RDP$_1$ is  a subdirect product of subdirectly irreducible kite pseudo effect algebras with \RDP$_1.$
\end{theorem}

\begin{proof}
Consider a kite $K_{I}^{\lambda,\rho}(G)$. By \cite[Thm 4.1]{DvuK}, $G$ satisfies RDP$_1.$ If it is not subdirectly
irreducible, then \cite[Thm 6.6]{DvuK} yields two possible cases:
(i) $G$ is not subdirectly irreducible, or
(ii) $G$ is subdirectly irreducible but
there exist $i,j\in I$ such that, for every
$m\in\mathbb N,$ we have $(\rho\circ\lambda^{-1})^m(i)\neq j$ and
$(\lambda\circ\rho^{-1})^m(i)\neq j$. Observe that this happens if and only if
$i$ and $j$ do not belong to the same connected component of $I$.

Now, using Lemma~\ref{le:7.3}
we can reduce (i) to (ii). So, suppose $G$ is subdirectly irreducible.
Then, using Lemma~\ref{le:7.4}, we can subdirectly embed
$K_{I}^{\lambda,\rho}(G)_{ea}$ into
$\prod_{I'} K_{J',I'}^{\lambda',\rho'}(G)_{ea}$, where $I'$
ranges over the connected components of $I,$ $J'=\lambda^{-1}(I')$ and $\lambda', \rho'$ are restrictions of $\lambda, \rho$ to $J'.$ But then, each
$K_{J',I'}^{\lambda,\rho}(G)_{ea}$ is subdirectly irreducible by
\cite[Thm 6.6]{DvuK} and by \cite[Thm 4.1]{DvuK}, it satisfies RDP$_1.$
\end{proof}

We note that if $G$ is an $\ell$-group, the corresponding kite pseudo effect algebra is a pseudo MV-algebra, and Theorem \ref{th:7.5} can be reformulated as follows, see \cite[Cor 5.14]{DvKo}.

\begin{theorem}\label{th:7.6}
The variety of pseudo MV-algebras generated by kite pseudo MV-algebras is generated by all subdirectly irreducible kite pseudo MV-algebras.
\end{theorem}

\section{Kite Pseudo Effect algebras and po-loops}

In this section, we show that every kite pseudo effect algebra is an interval in a po-loop which is not necessarily a po-group (i.e. it is not necessarily an associative po-loop). This construction is accompanied by illustrating examples.

We remind that according to \cite{Fuc}, a {\it po-groupoid} is an algebraic structure $(H;\cdot,\le~)$ such that (i) $(H;\cdot)$ is a groupoid, i.e. $H$ is closed under a multiplication $\cdot,$  and (ii) $H$ is endowed with a partial order $\le$ such that $a\le b$ entails $c\cdot a\cdot d \le c\cdot b\cdot d$ for all $c,d \in H.$  Sometimes for simplicity we will write $ab$ instead of $a\cdot b.$

If the multiplication $\cdot$ is associative, $(H;\cdot,\le)$ is said to be a {\it po-semigroup}. If $(H;\cdot)$ is a quasigroup, i.e.
for each $a$ and $b$ in $H$, there exist unique elements $x$ and $y$ in $H$ such that $a\cdot x=b$ and $y\cdot a=b,$ then $(H;\cdot,\le)$ is said to be a {\it po-quasigroup}. If in addition, a po-quasigroup $(H;\cdot,\le)$ has an identity element (or a neutral element) $e\in H$ such that $x\cdot e=x=e\cdot x$ for each $x \in H,$ then $(H;\cdot,e,\le)$ is said to be a {\it po-loop}. Let $x \in H$ be given. A unique element $x^{-1_r} \in H$ such that $x \cdot x^{-1_r}=e$ is said to be the {\it right inverse} of $x.$ Similarly, a unique element $x^{-1_l}\in H$ such that $x^{-1_l}\cdot x = e$ is said to be the {\it left inverse} of $x.$

Given an element $a$ of a groupoid $H,$ we define (i) $a^1 :=a,$ and (ii) if $n\ge 1,$ then $a^{n+1}:=a^na.$

As in the case of po-groups, an element $u\ge e$ is said to be a {\it strong unit} of a po-groupoid $H$ if  given an element $g \in H,$ there is an integer $n \ge 0$ such that $g \le u^n$; the couple $(H,u)$ is said to be a {\it unital po-groupoid}.

For a unital po-loop $(H,u),$ we define an interval $\Gamma(H,u)=\{h \in H\colon e\le h\le u\}.$

We know that every pseudo effect algebra with RDP$_1,$ \cite{DvVe1, DvVe2}, or every pseudo MV-algebra, \cite{Dvu2}, is an interval of a unital po-group with RDP$_1$ and of a unital $\ell$-group, respectively. A similar result, as we now show, is true also for every kite pseudo effect algebra, however, in this case it is an interval of a unital po-loop, which is not necessarily a po-group.

In what follows, we show that there are unital po-loops $(H,u),$ which are not necessarily associative, such that $\Gamma(H,u)$ can be endowed with a partial operation $+$ such that $(\Gamma(H,u);+,e,u)$ is a pseudo effect algebra, where $a+b$ for $a,b \in \Gamma(H,u)$ is defined in $\Gamma(H,u)$ whenever $ab\le u$ and then $a+b:= ab.$ We note that no assumption on the Riesz Decomposition Property of the kite pseudo effect algebra will be assumed.

\begin{theorem}\label{th:8.1}
For every kite pseudo effect algebra $K_{I}^{\lambda,\rho}(G)_{ea}$, there is a  unital po-loop $(H,u)$ such that $\Gamma(H,u)$ is a pseudo effect algebra and $K_{I}^{\lambda,\rho}(G)_{ea}$ is isomorphic to $\Gamma(H,u).$
\end{theorem}

\begin{proof}
Let $K_{I}^{\lambda,\rho}(G)_{ea}$ be a kite pseudo effect algebra of a po-group $G.$
We define $H:=W^{\lambda,\rho}_I(G):=\mathbb Z \lex G^I$ and let multiplication $*$ on $W^{\lambda,\rho}_I(G)$ be defined as follows

$$
(m_1,x_i)*(m_2,y_i):=(m_1+m_2, x_{\lambda^{-m_2}(i)}y_{\rho^{-m_1}(i)}).
$$
Then $W^{\lambda,\rho}_I(G)$ is a po-groupoid ordered lexicographically  such that $(0,(e))$ is the neutral element, $(m,x_i)^{-1_r}=(-m,x^{-1}_{(\rho^m\circ \lambda^m)(i)}),$ $(m,x_i)^{-1_l}=(-m,x^{-1}_{(\lambda^m\circ \rho^m)(i)}),$ and the element $u=(1,(e))$ is a strong unit. In fact, $W^{\lambda,\rho}_I(G)$ is a po-loop.

We note that the po-loop $W^{\lambda,\rho}_I(G)$ is associative iff $\lambda \circ \rho = \rho \circ \lambda,$ and in such a case, $W^{\lambda,\rho}_I(G)$ is a po-group.

An easy calculation shows that $(\Gamma(W^{\lambda,\rho}_I(G),u);+,e,u),$ where $e$ is an identity element of the po-loop $W^{\lambda,\rho}_I(G),$ is a pseudo effect algebra such that

\begin{align*}
(1,a_i^{-1})^\sim &= (0,a_{\rho(i)}),\quad
(1,a_i^{-1})^-= (0,a_{\lambda(i)}),\\
(0,f_i)^\sim &= (1,f^{-1}_{\lambda^{-1}(i)}),\quad
(0,f_i)^-= (1,f^{-1}_{\rho^{-1}(i)}).
\end{align*}

In addition, the pseudo effect algebra embedding of $K_{I}^{\lambda,\rho}(G)_{ea}$ onto  $\Gamma(W^{\lambda,\rho}_I(G),u)$ is defined as follows $\langle f_j\colon j \in I\rangle \mapsto (0,f_j)$ and $\langle a^{-1}_i\colon i \in I\rangle \mapsto (1,a^{-1}_i).$
\end{proof}

Now we show an example such that $\lambda\circ \rho \ne \rho \circ \lambda,$  $W^{\lambda,\rho}_I(G)$ is a non-associative po-loop, but $K_{I}^{\lambda,\rho}(G)_{ea}\cong\Gamma(W^{\lambda,\rho}_I(G),u).$

\begin{example}\label{ex:8.2}
Let $I=\{1,2,3,4\},$  $\lambda(1)=1,$ $\lambda(2)=3,$ $\lambda(3)= 2,$ $\lambda(4)=4,$ $\rho(1)= 2,$ $\rho(2)= 3,$ $\rho(3)= 1,$ $\rho(4)= 4.$ Then $\lambda \circ \rho\ne \rho\circ \lambda $ and $W^{\lambda,\rho}_I(G)$ is a non-associative po-groupoid, but $K_{I}^{\lambda,\rho}(G)_{ea}\cong\Gamma(W^{\lambda,\rho}_I(G),u).$
\end{example}

\begin{lemma}\label{le:8.3}
Let $I$ be a non-void set,  $\lambda, \rho\colon I \to I$ be bijections, and let there exist a decomposition $\{I_t\colon t \in T\}$ of $I,$ where each $I_t$ is non-void, such that $\lambda\circ\rho(I_t)=\rho\circ \lambda (I_t)$ for each $t \in T$ and, for each $t \in T$ there are $s_1,s_2 \in T$ such that $\lambda(I_t)=I_{s_1}$ and $\rho(I_t)=I_{s_2}.$ Then, for all integers $m,n \in \mathbb Z,$ we have $\lambda^m\circ \rho^n(I_t)=\rho^n\circ \lambda^m(I_t)$ for each $t \in T.$
\end{lemma}

\begin{proof}
(i) First we show that $\rho^{n+1}\circ \lambda(I_t)=\lambda\circ \rho^{n+1}(I_t)$ for each integer $n \ge 1$ and any $t \in T.$ Indeed, $\rho^2\circ \lambda(I_t)= \rho(\rho\circ \lambda(I_t))=\rho(\lambda\circ \rho(I_t))=\rho\circ\lambda(\rho(I_t))=\lambda\circ \rho(\rho(I_t))=\lambda\circ \rho^2(I_t).$

Now by induction we assume that $\rho^{i+1}\circ\lambda(I_s) = \lambda\circ \rho^{i+1}(I_s)$ holds for each $1\le i\le n$ and each $s \in T.$ Then $\rho^{n+1}\circ \lambda(I_t)= \rho^n(\rho\circ \lambda (I_t))=\rho^n(\lambda \circ \rho(I_t))= \rho^n\circ \lambda(\rho(I_t)) = \lambda\circ \rho^n(\rho(I_t))=\lambda\circ \rho^{n+1}(I_t).$

(ii) In the same way we have $\rho\circ \lambda^{m+1}(I_t)=\lambda^{m+1}\circ \rho(I_t)$ for each integer $m\ge 1$ and $t \in T.$

(iii) Using the analogous reasoning, we can show that $\lambda^m\circ \rho^n(I_t)=\rho^n\circ \lambda^m(I_t)$ for all integers $m,n\ge 1$ and each $t \in T.$

(iv) Using (iii), we have $I_t=\rho^{n}\circ \lambda^{m}(\lambda^{-m} \circ \rho^{-n}(I_t))= \lambda^m \circ \rho^n(\lambda^{-m} \circ \rho^{-n}(I_t))$ which yields $\rho^{-n}\circ \lambda^{-m}(I_t)= \lambda^{-m}\circ\rho^{-n}(I_t)$ for each $m,n\ge 0$ and each $t\in T.$

(v) By (iii), we have $\lambda^m\circ \rho^n(\lambda^{-m}(I_t))= \rho^n\circ \lambda^m(\lambda^{-m}(I_t))=\rho(I_t)$ which gives $\rho^n\circ \lambda^{-m}(I_t)=\lambda^{-m}\rho^n(I_t)$ for all $m,n \ge 0$ and each $t \in T.$

In the same way, we can prove $\rho^{-n}\circ \lambda^{m}(I_t)=\lambda^{m}\rho^{-n}(I_t)$ for $m,n \ge 0,$ $t\in T.$

Summarizing (i)-(v), we have proved the statement in question.
\end{proof}

We note that in Example \ref{ex:8.2}, we have $I=\{1,2,3\}\cup \{4\}$ and $\lambda(\{1,2,3\})=\{1,2,3\},$ $\rho(\{1,2,3\})=\{1,2,3\}$, $\lambda(\{4\})=\{4\}=\rho(\{4\})$ and the conditions of Lemma \ref{le:8.3} are satisfied.

\begin{example}\label{ex:8.4}
Let $I=\{1,2,3,4\},$  $\lambda(1)=1,$ $\lambda(2)=3,$ $\lambda(3)= 2,$ $\lambda(4)=4,$ $\rho(1)= 2,$ $\rho(2)= 1,$ $\rho(3)= 4,$ $\rho(4)= 3.$
If we set $I_1=\{1,4\},$ $I_2=\{2,3\},$ then $\lambda(I_1)=I_1,$ $\lambda(I_2)=I_2,$ $\rho(I_1)=I_2$ and $\rho(I_1).$ The conditions of Lemma {\rm \ref{le:8.3}} are satisfied.
\end{example}

\begin{example}\label{ex:8.5}
Let $I=\{1,2,3,4\},$  $\lambda(1)=2,$ $\lambda(2)=3,$ $\lambda(3)= 1,$ $\lambda(4)=4,$ $\rho(1)= 1,$ $\rho(2)= 3,$ $\rho(3)= 4,$ $\rho(4)= 2.$
If we set $I_1=\{1,4\},$ $I_2=\{2,3\},$ then $\lambda\circ\rho(I_1)= I_2=\rho\circ\lambda(I_1)$ and $\lambda\circ\rho(I_2)= I_1=\rho\circ\lambda(I_2).$ But $\lambda(I_1)=\{2,4\},$ $\lambda(I_2)=\{1,3\},$ $\rho(I_1)=\{1,2\}$ and $\rho(I_2)=\{3,4\}.$ Therefore, the decomposition $\{I_1,I_2\}$ does not satisfy the conditions of Lemma {\rm \ref{le:8.3}}, only the decomposition $\{I\}$ does.
\end{example}

\begin{proposition}\label{pr:8.6}
Let $I,\lambda,\rho$ and the decomposition $\{I_t\colon t \in T\}$ of $I$ satisfy the conditions of Lemma {\rm \ref{le:8.3}} and let $G$ be a po-group.  Let $H=\{(m,x_i)\in W^{\lambda,\rho}_I(G)\colon$ for each  $i,j \in I_t,\ x_i=x_j, \ t\in T\}.$ Then   $H$ is a po-group which is a po-subloop of $W^{\lambda,\rho}_I(G)$, and if $u:=(1,(e)) \in H,$ then $(H,u)$ is a unital po-group, and $\Gamma(H,u)$ is a subalgebra of the pseudo effect algebra $\Gamma(W^{\lambda,\rho}_I(G),u).$
\end{proposition}

\begin{proof}
The conditions of the proposition entail that $H$ is a subloop of the loop $W^{\lambda,\rho}_I(G)$ because the conditions imply by Lemma \ref{le:8.3}
$x_{\lambda^m\circ\rho^n(i)}=x_{\rho^n\circ \lambda^m(i)}$ for all $m,n \in \mathbb Z,$ $i \in I_t,$ $t \in T,$ that is,
associativity of the product holds in $H.$ Consequently, $\Gamma(H,u)$ is a pseudo effect algebra.
\end{proof}

\begin{remark}\label{re:8.7}
{\rm We note that if $\phi$ is the embedding from the proof of Theorem {\rm \ref{th:8.1}} of the kite pseudo effect algebra $K_{I}^{\lambda,\rho}(G)_{ea}$ onto the pseudo effect algebra $\Gamma(W^{\lambda,\rho}_I(G),u),$ then $\phi^{-1}(\Gamma(H,u))$ is a pseudo effect subalgebra of the kite pseudo effect algebra $K_{I}^{\lambda,\rho}(G)_{ea}.$}
\end{remark}

We note, in Example \ref{ex:8.5}, the only decomposition of $I$ satisfying Lemma \ref{le:8.3} is the family $\{I\}.$ Then the corresponding group $H$ from Proposition \ref{pr:8.6} consists only of the elements of the form $(m,(g)),$ where $(m,(g)):=(m,g_i)$ with $g_i = g,$ $i \in I$ $(g \in G),$ therefore, $\Gamma(H,u)\cong \Gamma(\mathbb Z\lex G,(1,0)).$

The same is true concerning the group $H$ from Proposition \ref{pr:8.6} for any case of $I$ and the decomposition $\{I\}.$

If we use the decomposition of $I$ consisting  only of the singletons of $I,$  $\{\{i\}\colon i \in I\},$ then Lemma \ref{le:8.3} holds for this decomposition iff $\lambda \circ \rho=\rho\circ \lambda,$ and for the po-group $H$ from Proposition \ref{pr:8.6}, we have $H=W^{\lambda,\rho}_I(G), $ and the po-loop $W^{\lambda,\rho}_I(G)$ is associative, that is, it is a po-group. Such a case can happen, e.g., if $\rho = \lambda^m$ for some integer $m \in \mathbb Z.$ But this is not a unique case as the following example shows also another possibility.

\begin{example}\label{ex:3.8}
Let $I=\{1,2,3,4\},$ $\lambda(1)=2,$ $\lambda(2)=1,$ $\lambda(3)=4,$ $\lambda(4)=3,$ and $\rho(1)= 4,$ $\rho(2)= 3,$ $\rho(3)= 2,$ $\rho(4)= 1.$ Then $\lambda\circ\rho =\rho\circ \lambda,$ $\lambda^2=id_I=\rho^2,$ but $\rho \ne \lambda^m$ for any $m \in \mathbb Z.$
\end{example}

In addition, if $H$ is an arbitrary subloop of the po-loop
$W^{\lambda,\rho}_I(G)$ such that $u=(1,(e)) \in H,$ then $\Gamma(H,u)$ is a pseudo effect subalgebra of the pseudo effect algebra $\Gamma(W^{\lambda,\rho}_I(G),u),$ and $\phi^{-1}(\Gamma(H,u))$ is also a pseudo effect subalgebra of the kite pseudo effect algebra $K_{I}^{\lambda,\rho}(G)_{ea}.$

From theory of pseudo effect algebras we know, \cite[Thm 7.2]{DvVe2}, that for every pseudo effect algebra $E$ with RDP$_1,$ there exists a unique (up to isomorphism) unital po-group $(H,u)$ with RDP$_1$ such that $E \cong \Gamma(H,u).$ From \cite[Thm 4.1]{DvuK}, it follows that if a po-group $G$ satisfies RDP$_1,$ then so does the kite pseudo effect algebra $K_{I}^{\lambda,\rho}(G)_{ea},$ and consequently, there is a unital po-group $(H,u)$ with RDP$_1$ such that $K_{I}^{\lambda,\rho}(G)_{ea}=\Gamma(H,u).$ Theorem \ref{th:8.1} guarantees the existence of a unital po-loop $(H,u)$ such that $K_{I}^{\lambda,\rho}(G)_{ea}\cong \Gamma(H,u),$ and by Examples \ref{ex:8.4}--\ref{ex:8.5}, we have seen that the po-loop $W^{\lambda,\rho}_I(G)$ is not associative.

\begin{problem}\label{prob:3.9}
{\rm (1)} In \cite{DvuK}, there was formulated an open problem to describe a unital po-group $(H,u)$ such that $K_{I}^{\lambda,\rho}(G)_{ea}\cong \Gamma(H,u).$ This is still not answered.

{\rm (2)} In addition, if there are two unital po-loops $(H_1,u_1)$ and $(H_2,u_2)$ such that $\Gamma(H_1,u_1)\cong K_{I}^{\lambda,\rho}(G)_{ea}\cong \Gamma(H_2,u_2),$ when are these po-loops  isomorphic to each other as po-loops~?
\end{problem}

\end{document}